\documentclass[]{amsart}
\usepackage{amssymb, mathdots}
\usepackage{mathrsfs}
\usepackage[utf8x]{inputenc}
\usepackage{amsmath, amsthm, graphicx}
\usepackage{diagrams}
\usepackage{wrapfig}
\usepackage{enumerate}
\usepackage{url}
\DeclareSymbolFont{bbold}{U}{bbold}{m}{n}
\DeclareSymbolFontAlphabet{\mathbbold}{bbold}
\def\qmod#1#2{{\hbox{}^{\displaystyle{#1}}}\!\big/\!\hbox{}_{
\displaystyle{#2}}}

\def\resto#1#2{{
#1\hskip 0.4ex\vline_{\hskip 0.2ex\raisebox{-0,2ex}
{{${\scriptstyle #2}$}}}}}

\def\C{{\mathbb C}}

\def\N{{\mathbb N}}

\def\P{{\mathbb P}}
\def\Q{{\mathbb Q}}
\def\R{{\mathbb R}}
\def\Z{{\mathbb Z}}

\def\union{\mathop{\bigcup}}

\def\map{\longrightarrow}
\def\textmap#1{\mathop{\vbox{\ialign{
                                  ##\crcr
      ${\scriptstyle\hfil\;\;#1\;\;\hfil}$\crcr
      \noalign{\kern 1pt\nointerlineskip}
      \rightarrowfill\crcr}}\;}}
\def\bigtextmap#1{\mathop{\vbox{\ialign{
                                  ##\crcr
      ${\hfil\;\;#1\;\;\hfil}$\crcr
      \noalign{\kern 1pt\nointerlineskip}
      \rightarrowfill\crcr}}\;}}
      
\newcommand{\cal}{\mathcal}
\def\textlmap#1{\mathop{\vbox{\ialign{
                                  ##\crcr
      ${\scriptstyle\hfil\;\;#1\;\;\hfil}$\crcr
      \noalign{\kern-1pt\nointerlineskip}
      \leftarrowfill\crcr}}\;}}

\def\dg{{\mathfrak d}}

\def\kg{{\mathfrak k}}

\def\qg{{\mathfrak q}}

\def\Bg{{\mathfrak B}}

\def\Ig{{\mathfrak I}}
\def\Jg{{\mathfrak J}}

\def\Sg{{\mathfrak S}}

\theoremstyle{remark}
\newtheorem{ex}{Example}[section]
\newtheorem{pb}{Problem}

\theoremstyle{plain}
\newtheorem{sz}{Satz}[section]
\newtheorem{thry}[sz]{Theorem}
\newtheorem{pr}[sz]{Proposition}
\newtheorem{co}[sz]{Corollary}
\newtheorem{dt}[sz]{Definition}

\newtheorem{re}[sz]{Remark}

\def\GL{\mathrm {GL}}

\def\Pic{\mathrm {Pic}}

\def\deg{\mathrm {deg}}

\def\st{\mathrm {st}}
\def\pst{\mathrm{pst}}

\def\ASD{\mathrm{ASD}}

\newcommand\smvee{{\hskip -0.3ex \raise 0.2ex\hbox{$\scriptscriptstyle\vee$}}}
\begin{document}
 
\title[Analytic cycles in flip passages and instanton moduli spaces]{Analytic cycles in flip passages and in instanton moduli spaces over non-Kählerian surfaces}
 \author{Andrei Teleman} 
 \thanks{The author has been partially supported by the ANR project MNGNK, decision 
Nr.  ANR-10-BLAN-0118}
\address{Aix Marseille Université, CNRS, Centrale Marseille, I2M, UMR 7373, 13453 Marseille, France }
\begin{abstract}

Let ${\cal M}^\st$ (${\cal M}^\pst$) be a moduli space of stable (polystable) bundles with fixed determinant on a complex surface  with $b_1=1$, $p_g=0$, and let $Z\subset {\cal M}^\st$ be a pure $k$-dimensional analytic set. We prove a general formula for the homological boundary $\delta[Z]^{BM}\in H_{2k-1}^{BM}(\partial\hat {\cal M}^\pst,\Z)$ of the Borel-Moore fundamental class of $Z$  in the boundary of the blow up moduli space   $\hat {\cal M}^\pst$. The proof is based on the holomorphic model theorem of \cite{Te5}, which identifies a neighborhood of a boundary component of $\hat {\cal M}^\pst$ with a neighborhood of the boundary of a ``blow up flip passage".

We then focus on a  particular instanton moduli space which intervenes  in our  program for proving the existence of curves on class VII surfaces. Using our result, combined  with general properties of the Donaldson cohomology classes, we prove incidence relations between the Zariski closures  (in the considered moduli space) of certain families of extensions.  These incidence relations are crucial for understanding the geometry of the moduli space, and cannot be obtained using classical complex geometric deformation theory.
\end{abstract}
\maketitle 
\section{Introduction} \label{intro}

Let $(X,g)$ be a Gauduchon surface \cite{Gau}, $(E,h)$ a Hermitian rank-2 bundle over $X$, ${\cal D}$ a holomorphic structure on the determinant line bundle $\det(E)$ and $a$ the Chern connection of the pair $({\cal D},\det(h))$. The moduli space ${\cal M}^\pst$ (${\cal M}^\st$) of polystable (stable) holomorphic structures ${\cal E}$ on $E$ with $\det({\cal E})={\cal D}$ can be identified with the instanton moduli space ${\cal M}^\ASD$ (respectively ${{\cal M}^\ASD}^*$) of (irreducible) projectively ASD unitary connections $A$ on $E$ with  $\det(A)=a$ \cite{DK},  \cite{Bu}, \cite{LT}, \cite{Te1}, \cite{Te3}, \cite{Te5}. The stable part ${\cal M}^\st \subset {\cal M}^\pst$ is open and has a natural complex space structure, which, in general, does not extend across the reduction locus ${\cal R}:={\cal M}^\pst\setminus {\cal M}^\st$ (the subspace of reducible instantons).

The set of topological decompositions of $E$ (as direct sum of line bundles) can be identified with the set
$${\cal D}ec(E):=\qmod{\{c\in H^2(X,\Z)|\ c(c_1(E)-c)=c_2(E)\}}{\sim}\ ,
$$
where $\sim$ is the equivalence relation defined by the involution $c\mapsto c_1(E)-c$. We assume that $c_1(E)\not\in 2H^2(X,\Z)$, which implies that this involution has no fixed points. Moreover, under this assumption,  we showed \cite{Te3}, \cite{Te5} that, for a surface with $b_1(X)=1$, $p_g(X)=0$ (in particular for   a class VII surface \cite{BHPV}, \cite{Na}), ${\cal R}$  decomposes as a disjoint union of circles
\begin{equation}\label{RecDec}
{\cal R}=\union_{\lambda\in {\cal D}ec(E)} C_\lambda\ ,
\end{equation}
where $C_\lambda:=\big\{[{\cal E}]\in {\cal M}^\pst|\ {\cal E}\hbox{ has a direct summand ${\cal L}$ with } c_1({\cal L})\in \lambda\big\}$. Choosing a representative $c\in \lambda$, and putting $\dg:=\frac{1}{2}\deg_g({\cal D})$, $C_\lambda$ can be identified with the circle 
$$C_c:=\{[{\cal L}]\in\Pic^c(X)|\ \deg_g({\cal L})=\dg\}$$
(see section \ref{BlUpSection}). Blowing up  ${\cal M}^\pst$ at a circle $C_\lambda$ of {\it regular} reductions yields a proper map $p_\lambda:\hat {\cal M}^\pst_\lambda\to {\cal M}^\pst$, defined on a space $\hat {\cal M}^\pst_\lambda$ which has a natural   structure of a manifold with boundary around the exceptional locus ${\cal P}_\lambda:=p_\lambda^{-1}(C_\lambda)$  \cite{Te5}. This  exceptional locus is a fiber bundle over $C_\lambda$ with a complex projective space as fiber.    It's important to point out that,  in our framework, a moduli space ${\cal M}^\st$ contains distinguished {\it locally} closed complex subspaces ${\cal P}^c_\varepsilon$, which correspond to families of stable extensions, and are described in Remark \ref{RemIntro} below. For   $[{\cal L}]\in\Pic(X)$ define 
$$\P_{[{\cal L}]}:=\{[{\cal E}]\in {\cal M}^\st|\ {\cal E}\hbox{ is an extension of ${\cal K}\otimes {\cal L}^\vee$ by ${\cal L}$}\}\ ,
$$
and for a subset $A\subset\R$,  and a class $c\in H^2(X,\Z)$, put 
$$\Pic^c(X)_A:=\{[{\cal L}]\in\Pic^c(X)|\ \deg_g({\cal L})\in A\}\ .$$
 When $A$ is a singleton (an open interval), $\Pic^c(X)_A$ is a circle  (an annulus).  
\begin{re}\label{RemIntro}  Suppose that  $b_1(X)=1$, $p_g(X)=0$,   $c_1(E)\not\in 2H^2(X,\Z)$, let $\lambda\in {\cal D}ec(X)$ such that $C_\lambda$ is a circle of regular reductions, and choose $c\in\lambda$. Then
\begin{enumerate}[1.] 
\item There exists an open neighborhood ${\cal U}_\lambda$ of $C_\lambda$ in ${\cal M}^\pst$ such that ${\cal U}_\lambda\cap{\cal M}^\st$  is a smooth complex manifold of  dimension $4c_2(E)-c_1^2(E)$,
\item There exists $\varepsilon>0$ such that 
\begin{enumerate}[i)]
\item for any    $[{\cal L}]\in \Pic^{c}(X)_{(\dg-\varepsilon,\dg)}$ the subspace
$\P_{[{\cal L}]}$ is a complex submanifold of ${\cal U}_\lambda\cap {\cal M}^\st$. This submanifold can be identified with $\P(H^1({\cal L}^{\otimes 2}\otimes {\cal K}^\vee))$,  where
$$\dim(H^1({\cal L}^{\otimes 2}\otimes {\cal K}^\vee)))=r_c:=-\frac{1}{2} (2c -c_1(E))(2c -c_1(E)+c_1(X))\ .$$
\item The union $\union_{[{\cal L}]\in \Pic^c(X)_{(\dg-\varepsilon,\dg)}}  \P_{[{\cal L}]}$ is disjoint, defines a smooth submanifold ${\cal P}^c_\varepsilon$ of ${\cal U}_\lambda\cap {\cal M}^\st$, and the natural map ${\cal P}^c_\varepsilon\to \Pic^c(X)_{(\dg-\varepsilon,\dg)}$ is a  fiber bundle with fiber $\P^{r_c-1}_\C$.
\end{enumerate}
\end{enumerate}
\end{re}
The following remark highlights an advantage of the blow up moduli space: 
\begin{re} With the notations, suppose $r_c>0$. For sufficiently small $\varepsilon>0$
\begin{enumerate}[1.] 
\item The closure $\tilde {\cal P}^c_\varepsilon$ of ${\cal P}^c_\varepsilon$ in $\hat {\cal M}^\pst_\lambda$ is a $\P^{r_c-1}_\C$-bundle  over $\Pic^c(X)_{[\dg-\varepsilon,\dg]}$,
\item The closure $\bar {\cal P}^c_\varepsilon$ of ${\cal P}^c_\varepsilon$ in ${\cal M}^\pst$ is obtained from $\tilde {\cal P}^c_\varepsilon$ be collapsing to points the fibers over the circle $\Pic^c_{\{\dg\}}$. \end{enumerate}
 
 \end{re}
Note that, in general, is very hard to describe  the Zariski closure of ${\cal P}^c_\varepsilon$ in ${\cal M}^\st$, even when $X$ is a known  surface (for instance a   Kato surface). Understanding the Zariski closures of these families, and the incidence relations between  these closures is very important for  understanding the geometry of the moduli space, and plays an important role in our program to prove existence of curves on class VII surfaces. This is why we are interested in the general properties of the analytic subsets of ${\cal M}^\st$. The main result of this article concerns the following 
\begin{pb} Let $Z\subset {\cal M}^\st$ be a pure $k$-dimensional analytic set, and let $[Z]^{BM}\in H_{2k}^{BM}({\cal M}^\st,\Z)$ its fundamental class in Borel-Moore homology. Determine the homological boundary $\delta([Z]^{BM})\in H_{2k-1}({\cal P}_\lambda,\Z)$ of $[Z]^{BM}$. 
\end{pb}

Why is this problem relevant for understanding the geometry of an instanton  moduli spaces ${\cal M}^\pst$ and its families of extensions?  In order to explain this, recall first \cite{DK} that ${\cal M}^\st={{\cal M}^\ASD}^*$ is naturally embedded in the infinite dimensional moduli space ${\cal B}^*$ of irreducible unitary connections $A$ on $E$ with $\det(A)=a$. The space ${\cal B}^*$ is endowed with tautological cohomology classes \cite{DK}, which will be called Donaldson classes.
In general, a Donaldson class $\nu\in H^*({\cal B}^*,\Q)$ cannot be extended across a circle of reductions, but it does extend to the exceptional locus ${\cal P}_\lambda\subset \hat {\cal M}^\pst_\lambda$ associated with a   circle $C_\lambda$ of  regular reductions. This a second important advantage of the blow up moduli space. With this remark we have
\newtheorem*{th-identity}{Proposition \ref{identity}} 
\begin{th-identity} 
Suppose that  $b_1(X)=1$, $p_g(X)=0$,   $c_1(E)\not\in 2H^2(X,\Z)$, ${\cal M}^\pst$ is compact, and all reductions in ${\cal M}^\pst$ are regular. For any Donaldson cohomology class $\nu\in H^{k-1}({\cal B}^*,\Q)$ and  any $\xi\in H_k^{BM}({\cal M}^\st,\Q)$ we have
\begin{equation}\label{id-formula}\sum_{\lambda\in {\cal D}ec(E)} \langle \resto{\nu}{{\cal P}_\lambda},\delta_\lambda\xi\rangle =0\ ,
\end{equation}
where $\delta_\lambda\xi$ denotes the homological boundary of $\xi$ in $H_{k-1}({\cal P}_\lambda,\Q)$.
\end{th-identity}

The point is that the restrictions $\resto{\nu}{{\cal P}_\lambda}$ have been computed explicitly   \cite[Corollary 2.6]{Te2}. Therefore, assuming that Problem 1 is solved, formula (\ref{id-formula}) yields a strong obstruction to the existence of an analytic set $Z\subset {\cal M}^\st$ with prescribed topological behavior around the circles of reductions $C_\lambda$. 
Note that (in the relevant cases) the particular moduli space ${\cal M}^\pst$, used in our program to prove the existence of curves on  class VII surfaces, satisfies the assumptions of Proposition \ref{identity} \cite{Te3}.
\vspace{1mm}

The main result of this article gives a solution to Problem 1. The proof is based on the holomorphic model theorem  \cite{Te5}, which states that a neighborhood of the boundary ${\cal P}_C$ in ${\cal M}^\pst$ can be identified with a neighborhood of the boundary of a standard model, which we called a {\it blow up flip passage}, and whose construction we recall briefly below.

Let $B$ be a Riemann surface, $p':E'\to B$, $p'':E''\to B$ holomorphic Hermitian bundles of ranks $r'$, $r''$, and $f:B\to\R$ a smooth function which is a submersion at any vanishing point, and such that $C:=f^{-1}(0)$ is a circle.  The direct sum $E:=E'\oplus E''$ is endowed with the $\C^*$-action $\zeta\cdot (y',y''))= (\zeta y',\zeta^{-1}y'')$. The zero locus $Z(m^f)$ of the map  $m^f:E\to \R$ defined by
$$m_b^f(y',y'')=\frac{1}{2}(\| y'\|^2-\|y''\|^2)+f(b)\ , \ \forall (y',y'')\in E_b \ ,
$$
is a smooth,  $S^1$-invariant hypersurface. The induced $S^1$-action on  $Z(m^f)$ is free away of $C$ (embedded in $E$ via the zero section). Put $F':=\resto{E'}{C}$, $F'':=\resto{E''}{C}$. The normal bundle of $C$ in 
$Z(m^f)$ can be identified (as an $S^1$-bundle) with $F'\oplus \bar F''$, hence the spherical blow up $\widehat{Z(m^f)}_C$ of $Z(m^f)$ at $C$ is a manifold with boundary, whose boundary can be identified with the sphere bundle $S(F'\oplus \bar F'')$ (see section \ref{BlUpSection} in this article, \cite{Te5}). The {\it blow up flip passage} $\hat Q_f$ associated with the data $(p':E'\to B', p'':E''\to B,f)$ is defined by
$$\hat Q_f:=\qmod{\widehat{Z(m^f)}_C}{S^1}\ .
$$
This quotient is a smooth manifold whose boundary $\partial\hat Q_f$ can be identified with the projective bundle $\P(F'\oplus \bar F'')$. The interior $\hat Q_f\setminus \partial\hat Q_f$ can be identified with the quotient $Q^s_f:=E^s_f/\C^*$, where $E^s_f:=\C^* \cdot(Z(m^f)\setminus C)$ is open in $E$. Therefore this interior comes with a natural complex structure. Note  also that $\hat Q_f$ comes with a map $\qg:\hat Q_f\to B$, whose restriction to $Q^s_f$ is a holomorphic submersion.  Choosing points $b_\pm\in B_\pm:=(\pm f)^{-1}(0,\infty)$ (and supposing $r'>0$, $r''>0$), the fibers $\qg^{-1}(b_\pm)$ are smooth complex manifolds related by a flip. This explains the choice of the terminology "flip passage". More precisely, put $P':=\P(E')$, $P'':=\P(E'')$ and let $P'_\pm$ ($P''_\pm$) be the restriction of $P'$ ($P''$) to $B_\pm$. The projective bundles $P'_-$, $P''_+$ are naturally embedded in $Q^s_f$. With these notations we see that the fiber $\qg^{-1}(b_+)$ is obtained from the fiber $\qg^{-1}(b_-)$ by ``replacing" $P'_{b_-}$ with $P''_{b_+}$. We can state now
\begin{pb} Let $Z\subset Q^s_f$ be a pure $k$-dimensional analytic set, and $[Z]^{BM}\in H_{2k}^{BM}(Q^s_f,\Z)$ its fundamental class in Borel-Moore homology. Determine the homological boundary $\delta([Z]^{BM})\in H_{2k-1}(\P(F'\oplus \bar F''),\Z)$ of $[Z]^{BM}$. 
\end{pb}
The two problems Problem 1, Problem 2, are related by the the holomorphic model theorem proved in \cite{Te5}, which we explain briefly below (see section \ref{HMT} for details). Coming back to our gauge theoretical framework, let $\lambda\in{\cal D}ec(E)$ with   a circle of regular reductions $C_\lambda$, and fix $c\in\lambda$. The holomorphic model theorem gives a system $(p'_c:E'_c\to B_c, p''_c:E''_c\to B_c, f_c)$ as above, a diffeomorphism $\Psi_c$ between a neighborhood $O_c$ of $\partial \hat Q_{f_c}$ in $\hat Q_{f_c}$ and a neighborhood ${\cal O}_\lambda$ of ${\cal P}_\lambda$ in $\hat {\cal M}^\pst_\lambda$, where $\hat Q_{f_c}$ is the blow up flip passage associated with $(p'_c:E'_c\to B_c, p''_c:E''_c\to B_c, f_c)$. Moreover, $\Psi_c$ induces a  diffeomorphism  $\partial \hat Q_{f_c}=\P(F'_c\oplus \bar F''_c)\textmap{\simeq} {\cal P}_\lambda$,  a biholomorphism $O_c\setminus \partial \hat Q_{f_c}\textmap{\simeq} {\cal O}_\lambda\setminus {\cal P}_\lambda$, and maps the projective bundles $P'_{c,-}$, $P''_{c,+}\subset Q^s_{f_c}$ onto the extension spaces ${\cal P}^c_\varepsilon$, ${\cal P}^{c_1(E)-c}_\varepsilon\subset {\cal M}^\st$ respectively. \vspace{2mm}  
We explain now our answer to {Problem 2}, which  concerns an arbitrary blow up flip passage $\hat Q_f$ (see section \ref{boundaryZ}). Let $\Theta'$, $\Theta''$ be the tautological line bundles of the projective bundles $P':=\P(E')$, $P'':=\P(E'')$, and denote by $Q$ the total space of the line bundle $p_1^*(\Theta')\otimes p_2^*(\Theta'')$ over the fiber product $P'\times_B P''$. We identify this fiber product with the zero section of $Q$. One has a natural biholomorphism $j_1:Q\setminus (P'\times_B P'')\to Q^s_f\setminus (P'_-\cup P''_+)$. 

\newtheorem*{th-partialZ}{Theorem \ref{partialZ}} 
\begin{th-partialZ} 
With the notations above, suppose $r'>0$, $r''>0$, and let  $Z\subset Q^s_f$ be an analytic subset of pure dimension $k\geq 1$ such that $\dim(Z\cap (P'_-\cup P''_+))<k$.  Then
\begin{enumerate} 
\item The closure $\tilde Z$ of $j_1^{-1}(Z\setminus (P'_-\cup P''_+))$ in $Q$ is a pure $k$-dimensional analytic subset of $Q$ with   $\dim (\tilde Z\cap(P'\times_B P''))<k$, 
\item Choosing a point $x\in C$, the equality
$$\delta[ Z]^{BM}=[C]\otimes  J^{E''_{x}}_{E'_{x}} \big([\tilde Z]^{BM} \cdot (P'_{x}\times P''_{x})  \big)$$
holds in $H_1(C,\Z)\otimes H_{2k-2}(\P(F'_{x}\oplus\bar F''_{x}),\Z)=H_{2k-1}(\P(F'\oplus\bar F''),\Z)$.
\end{enumerate} 
\end{th-partialZ}
In this theorem $J^{E''_{x}}_{E'_{x}}: H_{2s}(\P(E'_{x})\times \P(E''_{x}),\Z) \to H_{2s+2}(\P(E'_{x}\oplus\bar E''_{x}),\Z)$ denotes the join morphism defined in section \ref{joinsection} (see formula (\ref{HomoJoinNew}), Remark \ref{NewJoinOnGen}).

Our explicit applications concern the moduli space considered in our program  for proving the existence of curves on class VII surfaces \cite{Te3}. Let $(X,g)$ be a class VII surface endowed with a Gauduchon metric, ${\cal K}$  the canonical line bundle of $X$, $K$ its underlying differentiable line bundle,  $(E,h)$ a Hermitian rank 2-bundle on $X$ with $c_2(E)=0$ and $\det(E)=K$, and ${\cal M}^\st$, ${\cal M}^\pst$ the two moduli spaces  associated with the data $(X,g,E,h,{\cal K})$. ${\cal M}^\pst$ is always compact \cite[Theorem 1.11]{Te3}. Moreover, assuming that $\deg_g({\cal K})<0$, $X$ is minimal and is not an  Enoki surface, it follows that ${\cal M}^\st$ is a smooth manifold of dimension $b_2(X)$, and all the reductions in ${\cal M}^\pst$ are regular \cite[Theorem 1.3]{Te3}.   We suppose for simplicity that $H_1(X,\Z)=\Z$, which implies that $H^2(X,\Z)$ is torsion free. Let  $(e_i)_{1\leq i\leq b_2(X)}$ be a Donaldson basis of $H^2(X,\Z)$, i.e. a basis which satisfies the conditions
$$ e_i\cdot e_j=-\delta_{ij}\ , \ \sum_{i=1}^{b_2(X)} e_i=c_1({\cal K})\ .
$$
\cite[section 1.1]{Te3}. Put $\Ig:=\{1,\dots,b_2(X)\}$, and for $I\subset \Ig$ put
$$\bar I:=\Ig\setminus I\ ,\ e_I:=\sum_{i\in I} e_i\ .$$
Since $c_1(E)=\sum_{i\in\Ig} e_i$, we have ${\cal D}ec(E)=\big\{\{e_I,e_{\bar I}\}\big|\ I\subset\Ig\big\}$, hence ${\cal M}^\pst$ has $2^{b_2(X)-1}$ circles of reductions. Using the notation  introduced in Remark \ref{RemIntro}, we have $r_{e_I}=|\bar I|$. Therefore, putting $\kg:=\frac{1}{2}\deg_k({\cal K})$, we obtain,  for every $I\subsetneq \Ig$, a projective bundle ${\cal P}^{e_I}_\varepsilon\to \Pic^{e_I}(X)_{(\kg-\varepsilon,\kg)}$ with fiber $\P^{|\bar I|-1}_\C$, and an embedding of ${\cal P}^{e_I}$ in ${\cal M}^\st$.  In particular ${\cal P}^0_\varepsilon$ is open in ${\cal M}^\st$.   This open subset  defines an end (the ``known end") of ${\cal M}^\st$.  As explained before, in the general case, for a given index set $I\subsetneq \Ig$ (including for $\emptyset$) is very hard to describe  the Zariski closure of ${\cal P}^{e_I}_\varepsilon$ in ${\cal M}^\st$, even when $X$ is a known class VII surface. The result below, proved using the methods developed in this article, will be used in \cite{Te6}.

\newtheorem*{th-application}{Proposition \ref{application}} 
\begin{th-application}
Let $X$ be a minimal class VII surface with $H_1(X,\Z)\simeq\Z$ and $b_2(X)=3$, which is not an Enoki surface. Let $g$  be a Gauduchon metric on $X$ with $\deg_g({\cal K})<0$. Then
\begin{enumerate}[1.]
\item The component  ${\cal M}_0\in\pi_0({\cal M}^\pst)$ of ${\cal P}^0_\varepsilon$   contains all four circles of reductions,
\item Let $i\in\Ig$. The Zariski closure of ${\cal P}^{e_i}_\varepsilon$ in ${\cal M}^\st$ has pure dimension 2, does not intersect ${\cal P}^0_\varepsilon$,  and contains the curve ${\cal P}^{e_I}_\varepsilon$ for a subset $I\subset\Ig$ of cardinal 2.
\end{enumerate}
\end{th-application}

If a stable bundle ${\cal E}$ is the central term of an extension of the form
$$0\to {\cal L}\to {\cal E}\to {\cal K}\otimes {\cal L}^\vee\to 0
$$
with $c_1({\cal L})=e_I$, we agree to say that ${\cal E}$ is a stable extension of type $I$.
The first statement of Proposition \ref{application} shows that, for every $I\subsetneq \Ig$ there exists a family of stable extensions of type $\emptyset$ which converges in ${\cal M}^\st$ to a stable extension of type $I$. The second statement shows that for any $i\in\Ig$ there exists $I_i\subset\Ig$ of cardinal 2, and a family of extensions of type $\{i\}$ which converges in ${\cal M}^\st$ to a stable extension of type $I_i$. These statements cannot be proved using complex geometric arguments.

\section{Flip passages and the holomorphic  model theorem }

\subsection{The spherical blow up}\label{BlUpSection}

Let $M$ be an $m$-manifold,  $W\subset M$ a closed $r$-dimensional submanifold. The spherical blow up $\hat M_W$  of $M$ at $W$ is a manifold with boundary, which comes with a smooth map $p:\hat M_W\to M$ with the following properties:
\begin{enumerate}[1.]
\item The boundary $\partial \hat M_W$ of $\hat M_W$ coincides with the sphere bundle $\{N_W^M\}^*/\R_{>0}$, where $\{N_W^M\}^*$ denotes the complement of the zero section in the normal bundle $N_W^M$ of $W$ in $M$,
\item One has $p^{-1}(W)=\partial \hat M_W$, and the restriction $\resto{p}{\partial \hat M_W}:\partial \hat M_W\to W$ coincides with the bundle projection $\{N_W^M\}^*/\R_{>0}\to W$,
\item The restriction $\resto{p}{\hat M_W\setminus \partial \hat M_W}:\hat M_W\setminus \partial \hat M_W\to M\setminus W$ is a diffeomorphism.
\end{enumerate}

We refer to \cite{AK} for the construction and the functoriality properties of the spherical blow up. We will give now explicit constructions for the spherical blow up $\hat M_W$ in several situations which are of interest for us.
\\
\\
1. The case when $W=Z(s)$ for a regular section $s\in\Gamma(E)$.\vspace{1mm}

Let $p:E\to M$ be a   real rank $r$-bundle on $M$ and $s\in\Gamma(E)$  a section  which is regular (transversal to the zero section) at any vanishing point. 
In this case the zero locus $W:=Z(s)$ is a smooth codimension $r$  submanifold of $M$, and the intrinsic derivative of $s$ defines an isomorphism $N^M_W\to \resto{E}{W}$. The spherical blow up $\hat M_W$ can be obtained as follows: endow $E$ with an Euclidean structure, and let $\pi:S(E)\to M$ be the corresponding sphere bundle. Then $\hat M_W$ can be identified with the submanifold 
$$\hat s:=\{(y,\rho)\in S(E)\times[0,\infty)|\ \rho y=s(\pi(y))\}
$$
of $S(E)\times[0,\infty)$.  Note that the map $(y,\rho) \mapsto \rho y-s(\pi(y))$ is a section in the bundle $(\pi\circ p_1)^*(E)$ over $S(E)\times[0,\infty)$, and this section is transversal to the zero section of this pull back bundle. Via the identification $\hat M_W=\hat s$, the contraction map $\hat s\to M$ is the restriction $\resto{\pi\circ p_1}{\hat s}$.
\\ \\
2. The case when $M$ is the total space of a vector bundle $p:E\to B$, and $W=Z(\theta_E)$, where $\theta_E\in\Gamma(E,p^*(E))$ is the tautological section of $E$. \vspace{1mm}

The tautological section  $\theta_E\in\Gamma(E,p^*(E))$ of $E$ is defined by $\theta_E(y)=y$, where the right hand term is regarded as an element of the fiber $p^*(E)_y$. The zero locus of $\theta_E$ is the zero section $B\subset E$. This is a special case of 1., hence $\hat E_B=\hat \theta_E$. On the other hand $\hat \theta_E$ can be obviously identified with $S(E)\times[0,\infty)$ via the map $S(E)\times[0,\infty)\ni (y,\rho)\mapsto (\rho y,y,\rho)\in \hat \theta_E$. Therefore we get an identification
$$\hat E_B=S(E)\times[0,\infty)\ ,
$$
and, via this identification, the contraction map $S(E)\times[0,\infty)\to E$ is  $(y,\rho)\mapsto \rho y$.
\begin{ex} Consider the inclusion map  $j:Z(m^f)\hookrightarrow E$ of the submanifold intervening in the definition of a blow up flip passage (see section \ref{intro}, \cite{Te5}). Then $j^{-1}(B)=C$ and $j$ induces a bundle isomorphism $N_C^{Z(m^f)}=\resto{N^E_B}{C}$. This implies that the blow up $\widehat{Z(m^f)}_C$ can be identified with the preimage of $Z(m^f)$ in the blow up $\hat E_B$ of $E$ at the zero locus of its tautological section.  Since $\hat E_B=S(E)\times[0,\infty)$ 
\begin{equation}\label{hatZ}
\widehat{Z(m^f)}_C=\big\{ ((y',\bar y''),\rho)\in S(E'\oplus \bar E'')\times[0,\infty)|\ m^f(\rho y', \rho y'')=0\big\}\ .
\end{equation}
\end{ex}
{\ }\\ 
3.  The case when $M$ is the total space of a vector bundle $p:E\to B$, and $W=Z(\theta_E,p^*(s))$, where $s\in \Gamma(F)$ is a section of a bundle $q:F\to B$. \vspace{1mm}

Suppose $s\in \Gamma(F)$ is transversal to the zero section, and let $W:=Z(s)$ be its zero submanifold. Then $(\theta_E,p^*(s))\in \Gamma(p^*(E)\oplus p^*(F))$ is transversal to the zero section, $Z((\theta_E,p^*(s))$ can be identified with $W$ via $p$, and one has an obvious identification $N_W^E=\resto{(E\oplus F)}{W}$. In the same way as in 2. we get an identification
$$\hat E_W=\{((y,v),\rho)\in S(E\oplus F)\times[0,\infty)|\ \rho v= s(p(y))\}\ ,
$$ 
and, via this identification,  the contraction map $\hat E_W\to E$ is  $((y,v),\rho)\mapsto \rho y$.

\subsection{The holomorphic model theorem}\label{HMT}

 Let $X$ be a  surface with  $b_1(X)=1$, $p_g=0$. For such a surface the canonical morphism $H^1(X,\C)\to H^1(X,{\cal O}_X)$ is an isomorphism, hence  $\Pic^0(X)=H^1(X,{\cal O}_X)/H^1(X,\Z)\simeq \C^*$. 
 
 Fix $x_0\in X$, and let $\mathscr{L}$ be the Poincaré line bundle  normalized at   $x_0$ on the product $\Pic(X)\times X$. Let $p_1:\Pic(X)\times X\to \Pic(X)$, $p_2:\Pic(X)\times X\to X$ be the two projections.   The   $H^1(\Pic(X),\Z)\otimes H^1(X,\Z)$-term of the Künneth decomposition of $c_1(\mathscr{L})$ can be interpreted as a morphism    $\delta:H_1(X,\Z)\to H^1(\Pic(X),\Z)$.
 
 Fix a Gauduchon metric $g$ on $X$. For any $c\in H^2(X,\Z)$, $t\in\R$ the level set   $\Pic^c(X)_{\{t\}}:=\{\resto{\deg_g}{\Pic^c(X)}\}^{-1}(t)$ is a circle, and $\delta$ induces an isomorphism $H_1(X,\Z)\to H^1(\Pic^c(X)_{\{t\}},\Z)$. We will always endow   $\Pic^c(X)_{\{t\}}$ with the boundary orientation of $\partial \Pic^c(X)_{[t,+\infty)}$. With this convention we get an isomorphism $H_1(X,\Z)\to \Z$ given by $h\mapsto \langle \delta(h),[\Pic^c(X)_{\{t\}}]\rangle$, which is independent of $c\in H^2(X,\Z)$, $t\in\R$ and the Gauduchon metric $g$. This isomorphism defines a distinguished generator $\gamma_X\in H^1(X,\Z)$. Therefore, by definition, we have
 \begin{equation}\label{gamma}
 \langle \gamma_X,h\rangle =\langle \delta(h),[\Pic^c(X)_{\{t\}}]\rangle\,\ \forall h\in H_1(X,\Z)\ .
 \end{equation}

Let $E$  be a Hermitian rank 2-bundle  with $c_1(E)\not\in 2H^2(X,\Z)$ on $X$, ${\cal D}$ a  holomorphic structure on $\det(E)$,  $\lambda\in {\cal D}ec(E)$, and $c\in\lambda$. We define the harmonic map $f_c:\Pic^c(X)\to\R$ by
$f_c([{\cal L}]):=\pi\big(\deg_g({\cal L})-\dg),\hbox{ where }\dg:=\frac{1}{2}\deg_g({\cal D})$.
The vanishing circle $C_c:= f_c^{-1}(0)=\Pic^c(X)_{\{\dg\}}$ can be identified  with the reduction circle $C_\lambda\subset {\cal M}^\pst_{\cal D}(E)$ via the isomorphism $k_c:C_c\to C_\lambda$ given by $k_c([{\cal L}]):=[{\cal L}\oplus ({\cal D}\otimes {\cal L}^\vee)]$.
Using the Riemann-Roch and Grauert's   theorems we obtain   \cite[Proposition 3.3]{Te5}:
\begin{re} Suppose $C_\lambda$ is a circle of regular reductions. There exists a Zariski open neighborhood  $U$ of  $C_c$ in $\Pic^c(X)$ such that the restrictions of the coherent sheaves $R^1p_{1*}\big(\mathscr{L}^{\otimes 2}\otimes p_2^*({\cal D}^\vee)\big)$, $R^1p_{1*}\big(\mathscr{L}^{\otimes -2}\otimes p_2^*({\cal D})\big)$
to $U$  are locally free of ranks
$$r'_c=-\frac{1}{2} (2c -c_1(E))(2c -c_1(E)+c_1(X)),\ r''_c=-\frac{1}{2} (-2c +c_1(E))(-2c +c_1(E)+c_1(X))
$$
respectively, and the fibers of the corresponding bundles at a point $y\in U$ are identified with $H^1\big(\mathscr{L}^{\otimes 2}_{y}\otimes p_2^*({\cal D}^\vee)\big)$, $H^1\big(\mathscr{L}^{\otimes -2}_{y}\otimes p_2^*({\cal D})\big)$.
\end{re}

Therefore, for sufficiently small  $\varepsilon>0$, we obtain holomorphic bundles  $E'_c$, $E''_c$ of ranks $r'_c$, $r''_c$ on the annulus $B_c:=\Pic^c(X)_{(\dg-\varepsilon,\dg+\varepsilon)}=f_c^{-1}(-\pi\varepsilon,\pi\varepsilon)$.

The restrictions $\resto{E'_c}{C_c}$, $\resto{E''_c}{C_c}$ can be endowed with natural Hermitian metrics, by identifying them with suitable harmonic spaces \cite{Te5}. Using \cite[Remark 2.8]{Te5} we obtain a well defined blow up flip passage, which will be denoted by $\hat Q_{f_c}$. Correspondingly, we  put $F'_c:=\resto{E'_c}{C_c}$, $F''_c:=\resto{E''_c}{C_c}$, $P'_c:=\P(E'_c)$, $P''_c:=\P(E''_c)$, $P'_{c,\pm}:=\resto{P'_c}{B_{c,\pm}}$, $P''_{c,\pm}:=\resto{P''_c}{B_{c,\pm}}$.

  \begin{thry} \label{hol-model}
  
Under the assumptions and with the notations above there exists an open neighborhood $O_c$ of $\partial\hat Q$ in $\hat Q_{f_c}$ and a diffeomorphism $\Psi_c:O_c\to {\cal O}_\lambda$ onto a smooth open neighborhood ${\cal O}_\lambda$ of   ${\cal P}_\lambda$ in   $\hat{\cal M}^\pst_\lambda$ such that
\begin{enumerate}[1.] 
\item  $\Psi_c$ induces a smooth bundle isomorphism $\partial \Psi_c$ which fits in the commutative diagram\vspace{-2mm}
\begin{equation}\label{bundle-iso}\begin{diagram}[h=6mm]
\P (F'_c\oplus  \bar F''_c)= \partial\hat Q_{f_c}  &\rTo^{\partial \Psi_c} & {\cal P}_\lambda\\
\dTo & &\dTo_{\resto{p_\lambda}{{\cal P}_\lambda}}\\
C_c&\rTo^{k_c} & C_\lambda \ ,
\end{diagram}
\end{equation}
 
\item $\Psi_c$  induces a biholomorphism $O_c\setminus \partial\hat Q_{f_c}\to {\cal O}_\lambda\setminus {\cal P}_\lambda$,
\item $\Psi_c(P'_{c-})={\cal P}^c_\varepsilon\cap {\cal O}_\lambda$,    $\Psi_c(P''_{c+})={\cal P}^{c_1(E)-c}_\varepsilon\cap {\cal O}_\lambda$,
\item Denoting by $\nu(h)$, $\nu(u)$ the Donaldson classes associated with $h\in H_1(X,\Q)$, $u\in H_2(X,\Q)$, we have
$$(\partial\Psi_c)^*(\nu(h))=\delta(h)\otimes W_c\ ,\ (\partial\Psi_c)^*(\nu(u))=\frac{1}{2} \langle 2c-c_1(E),u\rangle W_c\ ,
$$
where $W_c$ is the positive generator of $H_2(\P(F'_c\oplus \bar F''_c),\Z)$.
\end{enumerate}
 \end{thry}
 
 The first two statements are proved in \cite{Te5}, and the third can be proved easily using the construction of $\Psi_c$. The fourth follows from \cite[Corollary 2.6]{Te2}.

\section{The homological boundary of an analytic set $Z\subset Q^\st_f$ }

    \subsection{The real blow up of a Borel-Moore homology class and its boundary}\label{RealBlowUpClass}

We start with a brief account on the well known intersection theory in Borel-Moore theory (see for instance \cite[section 1.12]{BH}) in a very particular framework. Since in our results the signs play a crucial role, and in the literature one can find  different conventions for the  relevant objects intervening in our formulae (the cap product, the orientation of the normal bundle of an oriented submanifold of an oriented manifold, the Thom isomorphism), we will write down carefully the  formulae we need.
\vspace{2mm}

In general, if $M$ is a differentiable manifold, and $Y\subset M$ a submanifold of $M$, we will always use the direct sum decomposition $\resto{T_M}{Y}=T_Y\oplus N_Y^M$ (defined by a Riemannian metric on $M$)  to orient anyone  of the three objects $M$, $Y$, $N_Y^M$ using orientations of the other two. In particular, the total space $E$ of an oriented Euclidean bundle over an oriented base $Y$ is oriented such that the obvious isomorphism $\resto{T_E}{Y}=T_Y\oplus E$ is orientation preserving. Note that \cite{BT} uses a different convention.
\vspace{2mm}

For an oriented $r$-dimensional Euclidean vector space $F$ we denote by $B(F)$ the open unit ball of $F$, by $[F,F\setminus B(F)]$ the generator of $H_r(F,F\setminus B(F),\Z)$ defined by the fixed orientation, and by $\{F,F\setminus B(F)\}$ the corresponding generator of $H^r(F,F\setminus B(F),\Z)$.

Let $p:E\to Y$ be an oriented Euclidean rank $r$-bundle over a connected, oriented, closed manifold $Y$. We denote by  $\Phi_E$ the section in the local coefficient system
$$x\mapsto H_r(E_x,E_x\setminus B(F_x),\Z)
$$
given by $x\mapsto [E_x,E_x\setminus B(F_x)]$, by  $B(E)$ the unit ball bundle of $E$, and by $\phi_E\in H^r(E,E\setminus B(E),\Z)$ its Thom class.   For any $k\in\N$ the morphism 
\begin{equation}\label{Thom}\tau_k^E:H_{r+k}(E,E\setminus B(E),\Z)\to H_k(Y,\Z)\ ,\ \tau_k^E(u):=(-1)^{kr}p_*(\phi_E\cap u)
\end{equation}
is an isomorphism \cite[Theorem 5.7.10]{Sp}. In this formula, on the right, we used the sign convention of \cite[Section VII.12]{Do} for the definition of the cap product, which does not agree with the conventions used in \cite{Sp}, \cite{BH}.
The factor $(-1)^{kr}$   has been inserted in order to recover the standard isomorphism
$$h\otimes [F,F\setminus B(F)]\mapsto h
$$
in the case of the trivial bundle $E=Y\times F\textmap{p_1} Y$. Having this case in mind, put
$$h\otimes \Phi_E:={\{\tau_k^E}\}^{-1}(h)\ .
$$ 
We will identify $H_*(F,F\setminus B(F),\Z)$, $H_*(E,E\setminus B(F),\Z)$ with the Borel-Moore homology $H_*^{BM}(B(F),\Z)$, $H_*^{BM}(B(E),\Z)$ of $F$ and $E$ respectively via the canonical isomorphisms
$$H_*(F,F\setminus B(F),\Z)\to H_*(\bar B(F),S(F),\Z)\to H_*^{BM}(B(F),\Z)\ ,
$$
$$H_*(E,E\setminus B(E),\Z)\to H_*(\bar B(E),S(E),\Z)\to H_*^{BM}(B(E),\Z)
$$
\cite[Section 1.2]{BH}. Via the first identification,  $[F,F\setminus B(F)]$ corresponds to the fundamental class $[B(F)]$ of the oriented manifold $F$ in Borel-Moore homology, hence $\Phi_E$ corresponds to the section $\Bg_E$ of the local coefficient system
$$x\mapsto H_r^{BM}(B(E_x),\Z)
$$
given by $x\mapsto [B(E_x)]$. Therefore the isomorphism $\tau_k^E$ induces an isomorphism $\tau_k^{B(E)}:H_{r+k}^{BM}(E,\Z)\to H_k(Y,\Z)$, and we can put 
$$h\otimes\Bg_E:=\{\tau_k^{B(E)}\}^{-1}(h)\ .
$$

By our orientation convention for $E$ (and implicitly $B(E)$), the fundamental class of $B(E)$ in Borel-Moore homology will be given by $[B(E)]=[Y]\otimes \Bg_E$.
 With these conventions, the morphism $p^*:H^*(Y,\Z)\to H^*(B(E),\Z)$ commutes with the Poincaré duality isomorphisms on $Y$ and $B(E)$, i.e. one has the identity
\begin{equation}\label{PDComm}
p^*(b)\cap [B(E)]=(b\cap [Y])\otimes \Bg_E\  \forall b\in H^*(Y,\Z)\ .
\end{equation}

\begin{dt} \label{HomInt} Let $M$ be a connected, oriented $m$-dimensional manifold. For a class $c\in H_k^{BM}(M,\Z)$ and a closed, oriented $l$-dimensional submanifold $W\subset M $ we define the homological intersection of   $c$ with $W$ by
$$c\cdot W:=j_W^*(b)\cap [W]\in H_{l+k-m}(W,\Z)\ ,$$
where $j_W:W\hookrightarrow M$ is the inclusion map, and $b\in H^{m-k}(M,\Z)$ is the Poincaré dual of $c$, i.e. one has $c=  b\cap [M]$.
\end{dt}  

\begin{re}\label{IntRem} In the conditions of Definition \ref{HomInt} it holds \begin{enumerate}[1.] 
\item  If $c$ is the fundamental class of a closed, oriented submanifold $Z\subset M$ which intersects $W$ transversally, then one has 
$$c\cdot W=(-1)^{(m-l)(m-k)}[W\cap Z]\ ,
$$
where $T=W\cap Z$ is regarded as a submanifold of $W$ endowed with the  orientation induced by the orientation of $W$, and the  orientation of $N_{W\cap Z}^W$ given by the natural isomorphism $N_{W\cap Z}^W=\resto{N_Z^M}{W\cap Z}$.
\item  If $M$ is a complex manifold, $W$ a complex submanifold, and $c$ the fundamental class of an  $k$-dimensional analytic subset $Z\subset M$ such that $W\cap Z$ has pure dimension $l+k-m$, then $c\cdot W$ is the fundamental class of the analytic cycle $W\cap Z$ of $W$.
\end{enumerate}
\end{re}

The unit sphere $S(F)$ of an oriented Euclidean space $F$ is given the boundary orientation of the closed unit disk $\bar B(F)$, i.e. the orientation for which the obvious isomorphism $\R \eta\oplus T_{S(F)}\simeq  \resto{T_F}{S(F)}$ is orientation preserving, where $\eta$ is the outer normal field of $S(F)$. The sphere bundle $S(E)$ of an oriented Euclidean bundle $p:E\to Y$ is oriented using the same rule {\it fiberwise}.  We denote by  $\Sg_E$ the section of the local coefficient system $x\mapsto H_{r-1}(S(E_x),\Z)
$ given by   $x\mapsto [S(E_x)]$. Using the Leray-Hirsch theorem again, and putting $n:=\dim(Y)$, we get as above an isomorphism 
$$H_{n+r-1}(S(E))\simeq H_n(Y)\otimes H_{r-1}(S^{r-1},\Z)
$$
and an identification $[S(E)]=[Y]\otimes \Sg_E$. Using the Borel-Moore long exact sequence
$$\dots\to H_j^{BM}(S(E),\Z) \to H_j^{BM}(\bar B(E),\Z)\to H_j^{BM}(B(E),\Z)\textmap{\partial} H_{j-1}^{BM}(S(E),\Z) 
$$
associated with the open embedding $B(E)\hookrightarrow \bar B(E)$ \cite[Section 1.6]{BH}, and comparing it with the long exact sequence  of the pair $(\bar B(E),S(E))$, we get 
\begin{equation}\label{dBE}
\partial(h\otimes \Bg_E)=(-1)^{s}h\otimes \Sg_E\ \forall h\in H_s(Y,\Z)\ .
\end{equation}
In particular $\partial [B(E)]=\partial([Y]\otimes \Bg_E)=(-1)^{n}[Y]\otimes \Sg_E=(-1)^{n}[S(E)]$. 
\vspace{2mm}

Let $M$ be an oriented Riemannian $m$-dimensional manifold and $W\subset M$ a closed, oriented  $l$-dimensional submanifold. Let $\hat M_W$ be the spherical blow up of $M$ with center $W$. {\it If we omit orientations}, the boundary $\partial M_W$ can be identified with the sphere bundle $S(N_W)$ of the normal bundle $N_W$ of $W$ in $M$.

Let    $c\in H_k^{BM}(M,\Z)$ be  a $k$-dimensional Borel-Moore homology class. The image $c^W:=c_{M\setminus W}$ of $c$ under the canonical morphism $H_k^{BM}(M,\Z)\to H_k^{BM}(M\setminus W,\Z)$ can be regarded as a $k$-dimensional Borel-Moore homology class of the interior  $\hat M_W\setminus \partial \hat M_W$ of the manifold with boundary $\hat M_W$. Our problem is to compute explicitly the boundary $\delta(c^W)\in H_{k-1}^{BM}(\partial\hat M_W,\Z)$ in terms of topological invariants of the triple $(M,W,c)$.\vspace{2mm}

Let $K$ be a compact tubular neighborhood of $W$, denote by $U$ its interior and by $c_U$, $c_{M\setminus K}$ the images of $c$ in $H_k^{BM}(U,\Z)$, $H_k^{BM}(M\setminus K,\Z)$ via the canonical morphisms. The Borel-Moore long exact sequence associated with the open embedding $U\cup (M\setminus K)\hookrightarrow M$ contains the segment
$$\dots\to H_k^{BM}(M,\Z)\to H_k^{BM}(U\cup (M\setminus K),\Z)\textmap{\delta} H_{k-1}(\partial K,\Z)\to\dots \ ,
$$
which shows that
$$\delta(c_U)+\delta(c_{M\setminus K})=0\ .
$$
Using the obvious identification $H_*^{BM}(\partial \hat M_W,\Z)=H_*^{BM}(\partial K,\Z)$, we obtain
$$\delta(c^W)=-\delta(c_U)\ .
$$

Write $c=  b\cap [M]$, where $b\in H^{m-k}(M,\Z)$ is the inverse image of $v$ via the Poincaré duality isomorphism. Denoting by $j_U:U\hookrightarrow M$, $j_W:W \hookrightarrow M$ the embedding maps,  by $p^U_W:U\to W$ the projection map, and using formula (\ref{PDComm}), we get 
$$c_U=j_U^*(b)\cap [U]=(p^U_W)^*(j_W^*(b))\cap [U]=(j_W^*(b)\cap [W])\otimes \Bg_{N_W}=(c\cdot W)\otimes \Bg_{N_W}\ .
$$
Therefore, by (\ref{dBE}) we get $\delta(c_U)=(-1)^{l+k-m} (c\cdot W)\otimes \Sg_{N_W}$.
This proves
\begin{pr} \label{BoundForm}
Under the assumptions and with the notations above one has
\begin{equation}\label{boundaryBlowUp}
\delta (c^W)=(-1)^{l+k-m+1} (c\cdot W)\otimes \Sg_{N_W}\ .
\end{equation}
\end{pr}
\begin{ex} For the special case $k=m$, $c=[M]^{BM}$ we have $b=1\in H^0(M)$ and $c\cdot W=[W]$, hence
$$\delta ([M]^W)=(-1)^{l+1} [W]\otimes \Sg_{N_W}\ .
$$
\end{ex}

 \subsection{The join morphisms $H_k(\P(U)\times \P(V),\Z)\to H_{k+2}(\P(U\oplus V),\Z)$}\label{joinsection} 
 
Let $G$ be a Lie group and $F$ a differentiable manifold.
 
\begin{dt} \label{Bundle} The $F$-bundle associated with a {\it left} principal $G$-bundle $P\to B$ and a left action  $\alpha:G\times F\to  F$ is defined by  
$P\times_\alpha V:=\qmod{P\times V}{G}
$
where $P\times V$ is endowed with the diagonal left $G$-action. 
\end{dt}
The standard example is the tautological line bundle $\Theta_U$ on the projective space $\P(U)$ of a   finite  dimensional complex vector space $U$. Denoting by $\chi_0$  the  standard  character $\C^*\to \C^*=\GL(1,\C)$ given by $\zeta\mapsto \zeta$, one has 
$$\Theta_U= U^*\times_{\chi^{-1}_0}\C\ .
$$ 
Let $U$, $V$ be finite dimensional complex vector spaces. For a subset $A\subset \P(U)\times \P(V)$ we define its set theoretical join $J_{UV}(A)\subset \P(U\oplus V)$ by
$$J_{UV}(A):=\bigcup_{([u],[v])\in A}\P(\C u\oplus\C v)\ .
$$ 
We chose the word ``join" because $\P(\C u\oplus\C v)$ is the line   joining the pair of points $[u,0]$, $[0,v]\in \P(U\oplus V)$, hence $J_{U,V}(A)$ is the union of  joining lines associated with pairs  $([u],[v])\in A$.  Consider the projective line bundle 
$$p_{UV}:\P(p_1^*(\Theta_U)\oplus p_2^*(\Theta_V))\to \P(U)\times\P(V)$$
 on   $\P(U)\times\P(V)$. Using Definition \ref{Bundle}, the total space $\P(p_1^*(\Theta_U)\oplus p_2^*(\Theta_V))$ can be identified with the associated bundle $(U^*\times V^*)\times_{\alpha}\P^1$, where $\alpha:(\C^*\times\C^*)\times\P^1\to \P^1$ acts by
\begin{equation}\label{alpha}
\alpha((\zeta_1,\zeta_2),[z_0,z_1])= [\zeta_1^{-1}z_0, \zeta_1^{-1}z_1]\ .
\end{equation} 
The map $q_{UV}:\P(p_1^*(\Theta_U)\oplus p_2^*(\Theta_V))\to \P(U\oplus V)$  given by
$$q_{UV}([(u,v),[z_0,z_1]]):=[z_0 u,z_1 v]\ \  \forall (u,v)\in U^*\times V^*  ,\ \forall [z_0,z_1]\in\P^1_\C
$$
 is a modification which blows up the linear subvarieties $\P(U\times\{0\})$, $\P(\{0\}\times V)$ of $\P(U\oplus V)$.  We obtain the diagram 
\begin{equation}\label{JoinDiagram}
\begin{diagram}[s=8mm,w=4mm,midshaft]
&&\P(p_1^*(\Theta_U)\oplus p_2^*(\Theta_V))&& \\
&\ldTo^{p_{{\scriptscriptstyle UV}}} &  &\rdTo^{q_{{\scriptscriptstyle UV}}} &\\
\P(U)\times\P(V)&&&&\P(U\oplus V)
 \end{diagram}\ .
\end{equation}
\\  
With these notations one has 
\begin{equation}\label{join}
J_{UV}(A)=q_{UV}(p_{UV}^{-1}(A))\ .
\end{equation}

This formula allows us to define a homological version of the join map.  The Gysin exact sequence \cite[Theorem 9.3.3]{Sp} of the sphere bundle 
$$p_{UV}:\P(p_1^*(\Theta_U)\oplus p_2^*(\Theta_V))\to \P(U)\times\P(V)$$
 contains the segment 
$$H_{2k+3}(\P(U)\times\P(V),\Z)\stackrel{\partial}{\to} H_{2k}(\P(U)\times\P(V),\Z) \textmap{\sigma_{UV}} H_{2k+2}(\P(p_1^*(\Theta_U)\oplus p_2^*(\Theta_V)),\Z)$$
$$\textmap{p_{UV*}} H_{2k+2}(\P(U)\times\P(V),\Z)\textmap{\partial} H_{2k-1}(\P(U)\times\P(V),\Z)\to \dots\ ,
$$
Since the odd dimensional homology of $\P(U)\times\P(V)$ vanishes, it follows that $\sigma_{UV}$ is a  monomorphism, more precisely identifies $H_{2k}(\P(U)\times\P(V),\Z)$ with 
$$\ker\big(p_{UV*}:H_{2k+2}(\P(p_1^*(\Theta_U)\oplus p_2^*(\Theta_V)),\Z)\to  H_{2k+2}(\P(U)\times\P(V),\Z)\big)\ .$$
 This morphism is defined using the homology spectral sequence associated with the the sphere bundle $p_{UV}$ in the following way (see the proof of \cite[Theorem 9.3.3]{Sp}). The local coefficient system
$$\P(U)\times\P(V)\ni ([u],[v])\to H^2(\P(\C u\oplus\C v))
$$
comes with a natural section $\Sg$ defined by the fundamental classes of the complex lines    $\P(\C u\oplus\C v)$, hence this local coefficient system can be written formally as $\Z\Sg$. The morphism $\sigma_{UV}$ is induced by the composition
$$E^2_{2k, 2}=H_{2k}(\P(U)\times\P(V),\Z\Sg)\to E^\infty_{2k, 2}\to H_{2k+2}(\P(p_1^*(\Theta_U)\oplus p_2^*(\Theta_V)),\Z)\ .
$$
 We define the morphism 
$J_{UV}: H_{2k}(\P(U)\times\P(V),\Z)\to H_{2k+2}(\P(U\oplus V),\Z)$ by
\begin{equation}\label{HomoJoin}
J_{UV}= (q_{UV})_*\circ \sigma_{UV}\ .
\end{equation}

The morphism  $\sigma_{UV}$ can be geometrically understood as follows. Let $M$, $N$ be closed, connected oriented manifolds and $\varphi:M\to \P(U)$, $\psi: N\to \P(V)$ smooth maps. The manifold $P_{\varphi,\psi}:=\P(\varphi^*(\Theta_U)\oplus \psi^*(\Theta_V))$ is a $\P^1_\C$ fiber bundle over $M\times N$, hence it comes with a canonical orientation induced by the complex orientations of the fibers and and the product orientation of the basis. Then 
\begin{re} \label{sigma} The class
$\sigma_{UV}(\varphi_*[M]\otimes\psi_*[N])$ coincides with the image of the fundamental class $[P_{\varphi,\psi}]$ under the natural map $P_{\varphi,\psi}\to \P(p_1^*(\Theta_U)\oplus p_2^*(\Theta_V))$.
\end{re}

Using Remark \ref{sigma}  one can check that the definition (\ref{HomoJoin}) is compatible with set theoretical join map defined above  restricted to the set of algebraic subvarieties of $\P(U)\times\P(V)$. More precisely, for an algebraic subvariety $A\subset \P(U)\times\P(V)$, one has
$$J_{UV}([A])=[q_{UV}(p_{UV}^{-1}(A))]\ ,
$$
where $[\ ]$ stands for the fundamental class of an analytic cycle. This gives
\begin{re}\label{JoinOnGen}
Using the Küneth isomorphism $H^*(\P(U)\times\P(V),\Z)=H^*(\P(U),\Z)\otimes H^*(\P(V),\Z)$ and denoting by $u_k$, $v_k$, $w_k$ the fundamental class of the $k$-dimensional linear variety in $\P(U)$, $\P(V)$, $\P(U\oplus V)$ respectively,  we have
$$J_{UV}(u_k\otimes v_l)=w_{k+l+1}\ .
$$
\end{re}

We need a version of the join morphism which takes values in  $H_{*}(\P(U\oplus \bar V),\Z)$. We denote by $\bar V$ the $\C$-vector space obtained by endowing $V$ with the scalar multiplication $\zeta\cdot_{\bar V} v:=\bar\zeta\cdot_V v$.
 The identity map $V\to \bar V$ becomes anti-linear and will be denoted by $v\mapsto \bar v$.  For a subset $A\subset \P(U)\times\P(V)$ we define  its set theoretical join $J_{U}^{V}(A)\subset \P(U\oplus\bar V) $ by
\begin{equation}\label{NewJoin} J_{U}^{V}(A):=\bigcup_{([u],[v])\in A}\P(\C u\oplus\C \bar v)\subset \P(U\oplus\bar V)\ .
\end{equation}
Correspondingly, we replace in diagram (\ref{JoinDiagram}) 
\begin{itemize}
\item the projective line bundle $p_{UV}:\P(p_1^*(\Theta_U)\oplus p_1^*(\Theta_V))\to \P(U)\times\P(V)$ by  
$$p_U^V:\P(p_1^*(\Theta_U)\oplus p_1^*(\bar\Theta_V))\to \P(U)\times\P(V)\ ,$$
 where $\bar\Theta_V$ is the line bundle $V^*\times_{\bar\chi_0^{-1}}\C^*$ on $\P(V)$. Therefore 
 \begin{equation}\label{with-alpha'}\P(p_1^*(\Theta_U)\oplus p_1^*(\bar\Theta_V))=(U^*\times V^*)\times_{\alpha'}\P^1\ ,
\end{equation}
 where $\alpha':(\C^*\times\C^*)\times\P^1\to\P^1$ acts by:
 \begin{equation}\label{alpha'}
\alpha'((\zeta_1,\zeta_2),[z_0,z_1])= [\zeta_1^{-1}z_0, \bar\zeta_2^{-1}z_1]\ .
\end{equation} 
 \item the map $q_{UV}$ by  the map
$$q_U^V:\P(p_1^*(\Theta_U)\oplus p_1^*(\bar\Theta_V))\to \P(U\oplus\bar V)\ ,$$
$$q_U^V([(u,v),[z_0,z_1]]):=[z_0 u,  z_1 \bar v] \ \forall (u,v)\in U^*\times V^*  ,\ \forall [z_0,z_1]\in\P^1_\C\ .$$
\end{itemize}
In this way we obtain a homological version of (\ref{NewJoin}), namely
\begin{equation}\label{HomoJoinNew}
J_{U}^{V}=q_{U*}^V\circ \sigma_U^V:H_{2k}(\P(U)\times\P(V),\Z)\to H_{2k+2}(\P(U\oplus\bar V),\Z) \ ,
\end{equation}
where  
$$\sigma_{U}^V:H_{2k}(\P(U)\times\P(V),\Z)\to H_{2k+2}(\P(p_1^*(\Theta_U)\oplus p_2^*(\bar\Theta_V)),\Z)$$
 is obtained in the same way as    $\sigma_{UV}$, but using the Gysin exact sequence of the sphere bundle $p_U^V$.
 
Denoting by $I_U^V:\P(U)\times\P(V)\to \P(U)\times\P(\bar V)$ the homeomorphism given by $([u],[v])\mapsto ([u],[\bar v])$, we have the identity
\begin{equation}\label{RElJoins}
J_{U}^{V}=J_{U\bar V}\circ (I_U^V)_*\ ,
\end{equation}
which yields following version of Remark \ref{JoinOnGen}:
\begin{re}\label{NewJoinOnGen}
Denoting by $u_k$, $v_k$, $w_k$ the fundamental class of the $k$-dimensional linear variety in $\P(U)$, $\P(V)$,  $\P(U\oplus \bar V)$ respectively,  and using the Küneth isomorphism we have
$$J_{U}^{V}(u_k\otimes v_l)=(-1)^lw_{k+l+1}\ .
$$
\end{re}

\subsection{The closure in $\hat Q_f$ of an analytic set $Z\subset Q^s_f$ and the  boundary of its fundamental class} \label{boundaryZ}

We come back to the blow up flip passage $\hat Q_f$ associated with the data $(p':E'\to B, p'':E''\to B,f)$ as explained in section \ref{intro}.
 
Let $Z\subset Q^s_f$ be an analytic set of pure dimension $k>0$, and let $[Z]^{BM}\in H_{2k}^{BM}( Q^s_f, \Z)$ be the fundamental class of $Z$ in the Borel-Moore homology of $Q^s_f$. The long exact sequence associated to the open embedding $Q^s_f\hookrightarrow \hat Q_f$ contains the segment
$$\dots\to H_{2k}^{BM}(\hat Q_f, \Z) \to H_{2k}^{BM}( Q^s_f, \Z)\textmap{\delta} H_{2k-1}^{BM}( \partial\hat Q_f, \Z) \to\dots
$$
We will give an explicit formula, in complex geometric terms, for the   boundary 
$$\delta([ Z]^{BM})\in H_{2k-1}^{BM}( \partial\hat Q_f, \Z)=H_{2k-1}(\P(F'\oplus\bar F''), \Z)\ .$$

The problem is difficult because, a priori, one has no control on the closure of $Z$ in the manifold with boundary $\hat Q_f$.  Our result will implicitly give an explicit construction of this closure, and show that it belongs to a very special class of real analytic sets of  $\hat Q_f$.

 We recall that the circle $C=f^{-1}(0)$ has been endowed with the orientation defined by the complex orientation of $B$ and the orientation of $N_C^B$ defined by the differential $df$. Equivalently, $C$ is endowed with the boundary orientation of   $\partial f^{-1}([0,\infty))$. The bundles $F'$, $F''$ on $C$ are trivial hence, choosing a point $x\in C$ we get a canonical isomorphism  $H_*(\P(F'\oplus\bar F''), \Z)=H_*(C,\Z)\otimes H_*(\P(F'_{x}\oplus\bar F''_{x}), \Z)$. In particular, for $1\leq k\leq r'+r''$ one has
$$H_{2k-1}(\P(F'\oplus\bar F''), \Z)=[C]\otimes \Z w_{k-1}\ ,
$$
where $w_{k-1}$  denotes the standard generator of   $H_{2k-2}(\P(F'_{x}\oplus\bar F''_{x}), \Z)$.

Note first that, by definition, $Q^s_f$ contains two obvious complex submanifolds for which the problem can be solved easily: Putting $P':=\P(E')$, $P'':=\P(E'')$, $B_\pm:=(\pm f)^{-1}(0,\infty)$,   $E'_\pm:=\resto{E'}{B_\pm}$, $E''_\pm:=\resto{E''}{B_\pm}$, $P'_\pm:=\resto{P'}{B_\pm}$, $P''_\pm:=\resto{P''}{B_\pm}$ (as in section \ref{intro}), we see that $P'_-$, $P''_+$ are analytic subsets of $Q^s_f$.  
\begin{re}\label{P'P''} Under the assumptions and with the notations above one has 
$$\delta([P'_-]^{BM})=- [C]\otimes w_{r'-1} \hbox{ if }r'>0\ ,\ \delta([P''_+]^{BM})=(-1)^{r''-1} [C]\otimes w_{r''-1}  \hbox{ if }r''>0.
$$
\end{re}

The first (second) formula also applies in the special case when $r'=0$  ($r''=0$); in this special case we have $P'_-=\emptyset$, $P''_+=Q^s_f$ (respectively $P''_+=\emptyset$, $P'_-=Q^s_f$). 

 The union $P'_-\cup P''_+$  is disjoint, and should be regarded as the {\it known} part of $Q^s_f$. 
The main result of this section concerns the  more difficult case of a pure $k$-dimensional analytic set   $Z$ which is generically exterior  to this ``known" part of $Q_f^s$, in the sense that   $\dim(Z\cap (P'_-\cup P''_+))<k$.  

Let $E'_*$, $E''_*$ be the complements of the zero sections in the bundles $E'$, $E''$. Regarding $E'_*\times_ B E''_*$ as a {\it left} principal $\C^*\times\C^*$-bundle over $P'\times_B P''$, consider the associated line bundle given by
$$q:Q:=(E'_*\times_ B E''_*)\times_{(\chi_1^{-1}\chi_2^{-1})}\C\map P'\times_B P''\ ,
$$ 
where $\chi_i:\C^*\times\C^*\to\C^*$ is the character defined by the projection on the $i$-th factor. In other words, denoting by $\Theta'$, $\Theta''$ the tautological line bundles on the projective bundles $P'$, $P''$ and by $p_1:P'\times_BP''\to P'$, $p_2:P'\times_BP''\to P''$ the two projections,  one has $Q=p_1^*(\Theta')\otimes p_2^*(\Theta'')$ as line bundles over $P'\times_BP''$. The zero section  of this line bundle is a smooth hypersurface which can be identified with $P'\times_B P''$.

We denote  by $Q_0\subset E'\otimes E''$ the image of the natural map $E'\times_B E''\to  E'\otimes E''$. This map  is obviously $\C^*$-invariant. $Q_0$ is a locally trivial fiber bundle over $B$ whose fiber  over $x\in B$ is the cone over the image of $P'_x\times P''_x$ via the Segre embedding $P'_x\times P''_x\to \P(E'_x\otimes E''_x)$.  This cone   is singular when   $r'>1$ and $r''>1$, hence in this case $Q_0$ is a singular complex space  whose singular locus is the zero section $B\subset E'\otimes E''$. 
One has an obvious contraction $c:  Q\to Q_0$ which induces a biholomorphism $  Q\setminus (P'\times_B P'')\to Q_0\setminus B$, and contracts  {\it fiberwise} the divisor $P'\times_B P''$ to the zero section $B$ of $Q_0$.  In other words $c$ is the modification with center $B$, and the fiber product $P'\times_B P''$ is the exceptional divisor of the modification $c$. In the diagram below all polygons are commutative,  $j_1$, $j_2$ are biholomorphisms. The maps  $c_1$, $c_2$ are holomorphic modifications, in particular they are proper. The map $c_1$ contracts $P'_x\times P''_x$ to $P'_x$ when $x\in B_-$ and  contracts $P'_x\times P''_x$ to $P''_x$ when $x\in B_+$. The map $c_2$ contacts $P'_x\subset Q^s_f$ to $\{x\}$ when $x\in B_-$, and contacts $P''_x\subset Q^s_f$ to $\{x\}$ when $x\in B_+$.  Note that 
\begin{equation}\label{c2c1}
c_2\circ c_1=\resto{c}{Q\setminus (P'_C\times_C P''_C)}\ .
\end{equation}
$$
\begin{diagram}[h=6mm] 
&&  Q\setminus (P'\times_B P'') & \rTo^{j_1=} & Q^s_f\setminus (P'_-\cup P''_+) & &\rTo^{j_2=} & & Q_0\setminus B\\
&&&&&\luInto &&\ruInto\\
&&&&&& Z\setminus (P'_-\cup P''_+)& &  \\
&&\dInto &  & \dInto & &\dInto &&\dInto\\
&&&&&& Z\\
&&&&&\ldInto\\
&&  Q\setminus (P'_C\times_C P''_C) & \rTo^{c_1} &  Q^s_f &  & \rTo^{c_2}&&Q_0\setminus C
\end{diagram}$$
Using the notations introduced in sections \ref{RealBlowUpClass}, \ref{joinsection} we can state now
\begin{thry} \label{partialZ} 
With the notations above, suppose $r'>0$, $r''>0$, and let  $Z\subset Q^s_f$ be an analytic subset of pure dimension $k\geq 1$, such that $\dim(Z\cap (P'_-\cup P''_+))<k$.  Then
\begin{enumerate} 
\item The closure $\tilde Z$ of $j_1^{-1}(Z\setminus (P'_-\cup P''_+))$ in $Q$ is a pure $k$-dimensional analytic subset of $Q$ with   $\dim (\tilde Z\cap(P'\times_B P''))<k$, 
\item Choosing a point $x\in C$, the equality
$$\delta[ Z]^{BM}=[C]\otimes  J^{E''_{x}}_{E'_{x}} \big([\tilde Z]^{BM} \cdot (P'_{x}\times P''_{x})  \big)$$
holds in $ H_1(C,\Z)\otimes H_{2k-2}(\P(F'_{x}\oplus\bar F''_{x}),\Z)=H_{2k-1}(\P(F'\oplus\bar F''),\Z)$.
\end{enumerate} 
\end{thry}

The second statement gives an explicit formula for the boundary $\delta[  Z]^{BM}$ in complex geometric terms: this class is determined by the  homological intersection (see Definition \ref{HomInt}) of the Borel Moore class of $[\tilde Z]$ with the complex submanifold $P'_{x}\times P''_{x}$  of $Q$.
\begin{proof} (1) Since we assumed $\dim(Z\cap (P'_-\cup P''_+))<k$, it follows that $Z$ coincides with the closure of  $Z\setminus (P'_-\cup P''_+)$ in $Q^s_f$. We will apply Remmert-Stein theorem \cite[sec. 2 p. 354]{GR}   to the inclusion
$$j_2(Z\setminus (P'_-\cup P''_+)) \hookrightarrow Q_0\setminus B\ .
$$
This theorem applies automatically when $k\geq 2$ and shows that the closure $Z_0$ of $j_2(Z\setminus (P'_-\cup P''_+))$ in $Q_0$ is a complex analytic set of dimension $k$. In the limit case $k=1=\dim(B)$ we have to check that this closure $Z_0$  does not contain all of $B$. Since, by assumption, $\dim(Z\cap (P'_-\cup P''_+))<1$, it follows that $A:=Z\cap (P'_-\cup P''_+)$ is a 0-dimensional analytic subset  of $P'_-\cup P''_+$.
We claim that 
$$Z_0\cap (B\setminus C)\subset   c_2(A)\ ,$$
 which will complete the argument. To prove this inclusion, let  $x \in B\setminus C$ be the limit (in $Q_0$) of a sequence $(j_2(z_n))_{n\in \N}$ of points of $j_2(Z\setminus (P'_-\cup P''_+))$.   Since $c_2$ is proper,  a subsequence $(z_{n_k})_{k\in\N}$ of $(z_n)_{n\in \N}$ will converge to a point in $y\in P'_-\cup P''_+$ with $c_2(y)=x$. But $Z$ is closed in $Q^s_f$, hence we will have $y\in Z$, hence $x\in c_2(A)$ as claimed.   

Since $c$ is a modification, the closure $\tilde Z$ of $c^{-1}(Z_0\setminus B)=j_1^{-1}(Z\setminus (P'_-\cup P''_+))$ in $Q$ is an analytic set.  This analytic set is the proper transform $\tilde Z$ of $Z_0$ in $Q$, and it   is pure $k$-dimensional, because $Z_0\setminus B$ has this property.    Note also that $\dim(\tilde Z\cap (P'\times_B P''))<k$ as claimed, because the complement of  $\tilde Z\cap (P'\times_B P'')$ in $\tilde Z$ (which is obviously Zariski open in $\tilde Z$) is dense in $\tilde Z$. 
\vspace{2mm}\\
(2) Denote by $P'_C$, $P''_C$ the restrictions of $P'$, $P''$ to $C$, and put $W:=P'_C\times_C P''_C$ regarded a submanifold of $Q$ of real codimension 3.   As a subset of  the smooth hypersurface $P'\times_B P''\subset Q$ (the  zero section of $Q$), $W$ is the  smooth real hypersurface    defined by the equation $f\circ \pi=0$, where $\pi:P'\times_B P''\to B$ is the obvious projection.  The idea of the proof is to construct a {\it continuous} extension   $\hat c:\hat Q_W\to \hat Q_f$ of $c_1$  and to make use of Proposition \ref{BoundForm}  applied to   $(Q,W,[\tilde Z]^{BM})$ in order to compute the boundary of   $[\tilde Z\setminus W]^{BM}$ in the homology of $\partial\hat Q_W$. 
$$
\begin{diagram}[h=6mm]
&&  &  & \tilde Z&\rTo &Z_0&& \\
&&&\ldInto&&&&\rdInto\\
W=P'_C\times_C P''_C&\rInto & Q& & &\rTo^c  & &     &  Q_0\\
\uTo &&\uTo &   &&  &&  & \\
\partial \hat Q_{W}= S(N_W^Q)&\rInto& \hat Q_{W}&\rTo^{\hat c}&\hat Q_f&\lInto &\P(F'\oplus\bar F'')=&\partial \hat Q_f&  
\end{diagram}
$$
Note first that $W$ is the zero locus of the section $(\theta_Q,f\circ\pi\circ q)$ in $q^*(Q)\oplus\underline{\R}$, where $\theta_Q$ is the tautological section of the complex line bundle $q^*(Q)$ over $Q$. Therefore, as explained  section \ref{BlUpSection}, the blow up $\hat Q_W$ can be identified with the submanifold
$$\{((y,t),\rho)\in S(Q\oplus \underline{\R})\times[0,\infty)|\ \rho t= f( \pi(q(y))) \}
$$
of  $S(Q\oplus \underline{\R})\times[0,\infty)$, and, via this identification, the contraction map  $\hat Q_W\to Q$ is given by $((y,t),\rho)\mapsto \rho y$. Note that the sphere bundle  $S(Q\oplus \underline{\R})$ is the associated bundle
$$(S(E')\times_B S(E''))\times_{\beta} S(\C\oplus\R)\ ,
$$
where $S(E')\times_B S(E'')$ is regarded as a left principal $S^1\times S^1$-bundle over 
$P'\times_B P''$, and the action $\beta:(S^1\times S^1)\times S(\C\oplus\R)\to S(\C\oplus\R)$ is given by 
\begin{equation}\label{beta}
\beta((\zeta_1,\zeta_2),(u,t))= (\zeta_1^{-1}\zeta_2^{-1} u,t)\ .
\end{equation}
Using formula (\ref{hatZ}) of section \ref{BlUpSection}, one has an  identification
$$\hat Q_f=\big\{ ([y',\bar y''],\rho)\in    \{{S(E'\oplus \bar E'')}/{S^1}\}\times[0,\infty)|\ m^f(\rho y', \rho y'')=0\big\}\ ,
$$
and, via this identification, the contraction map  $\hat Q_f\to Q_f$ is   $([y',y''],\rho)\to [\rho y', \rho y'']$. 

We define $c_f:\hat Q_W\to \hat Q_f$ by
\begin{equation}\label{hatc} \hat c([(a',a''),(u,t)],\rho):=\left\{
\begin{array}{ccr}
\big(\frac{1}{\sqrt{2}}\big[(1-t)^{\frac{1}{2}}\ a', (1-t)^{-\frac{1}{2}}\bar u \ \bar a''\big]  ,\sqrt{2\rho}\big) &{\rm  for} & t\ne\phantom{-}1 \\
\big(\frac{1}{\sqrt{2}}\big[u(1+t)^{-\frac{1}{2}}\ a', 
(1+t)^{\frac{1}{2}}\ \bar a''\big],\sqrt{2\rho}\big)
&{\rm for }& t\ne -1
\end{array} 
\right. \ ,
\end{equation}
for any $(a',a'')\in S(E')\times_B S(E'')$, $(u,t)\in S(\C\oplus\R)$, $\rho\in[0,\infty)$. It is easy to see that $\hat c$ is well defined, takes values in $\hat Q_f$, and is a continuous extension of the holomorphic modification $c_1$. 
 
The boundary of $[\tilde Z\setminus W]^{BM}$ in $H_{2k-1}(\partial \hat Q_W,\Z)$ can be computed easily using Proposition \ref{BoundForm}, and the result is:  
\begin{equation}\label{FirstFormula}\delta([\tilde Z\setminus W]^{BM})=([\tilde Z]^{BM}\cdot W)\otimes \Sg_{N_W^Q}\ ,
\end{equation}
where $[\tilde Z]^{BM}\cdot W\in H_{2k-3}(W,\Z)$ is given by Definition \ref{HomInt}, and $\Sg_{N_W^Q}$ is the section   defined by the fundamental classes of the spheres  $S(N^Q_{W,w})$, $w\in W$.  In this formula $N_W^Q$ is oriented using the decomposition  $N_W^Q=\resto{Q}{W}\oplus \underline{\R}$, and the complex orientation of the complex line bundle $Q$. Correspondingly,  $W$ is oriented using this orientation of $N_W^Q$, and the complex orientation of the manifold $Q$. This orientation of $W$ agrees with its orientation as $\P^{r'}_\C\times \P^{r''}_\C$-bundle over the oriented circle $C$.

Under our assumptions $W$ is a $\P^{r'}_\C\times \P^{r''}_\C$ bundle over $C$ (which is a  circle), hence any cohomology class $b$  of even degree  of $W$ can be written as $1\otimes j_{x}^*(b)$, where $j_{x}:P'_{x}\times P''_{x}\hookrightarrow W$ denotes the inclusion map. Therefore, for any Borel-More homology class $c$ of even degree of $Q$, denoting by $b$ its Poincaré dual, we get
$$c\cdot W=j_W^*(b)\cap [W]=\big(1\otimes j_{x}^*j_W^*(b)\big)\cap \big([C]\otimes [P'_{x}\times P''_{x}]\big)=[C]\otimes \big (c\cdot (P'_{x}\times P''_{x})\big)\ . 
$$
In particular
$$[\tilde Z]^{BM}\cdot W=[C]\otimes \big([\tilde Z]^{BM}\cdot(P'_{x}\times P''_{x})\big)\ ,
$$
so that, by (\ref{FirstFormula}) we get
\begin{equation}\label{Second}\delta([\tilde Z\setminus W]^{BM})=[C]\otimes \big([\tilde Z]^{BM}\cdot(P'_{x}\times P''_{x})\big)\otimes \Sg_{N_W^Q}\ .
\end{equation}
The morphism of triples $(\hat Q_W\setminus \partial \hat Q_W,\hat Q_W,\partial \hat Q_W)\to (\hat Q_f\setminus\partial \hat Q_f, \hat Q_f, \partial \hat Q_f)$ induced by $\hat c$ gives a morphism of Borel-Moore homology long  exact sequences. Since %
$$\resto{c_1}{\tilde Z\setminus W}:\tilde Z\setminus W\to Z$$
is a modification, we have $(c_{1*})([\tilde Z\setminus W]^{BM})=[Z]^{BM}$, which implies
$$\delta([Z]^{BM})=(\partial \hat c)_*(\delta([\tilde Z\setminus W]^{BM}))\ .
$$
The two boundaries $\partial \hat Q_W$, $\partial \hat Q_f$ are fiber bundles over $C$, and the map 
$$\partial \hat c_{x}:(S(E'_{x})\times  S(E''_{x}))\times_{\beta} S(\C\oplus\R) \to \P(E'_{x}\oplus\bar E''_{x})$$
induced by $\partial \hat c$ is a bundle morphism over $C$. Therefore, writing   $\delta([Z]^{BM})$ as $\delta([Z]^{BM})=[C]\otimes U$ for a class $U\in H_{2k-2}(\P(E'_{x}\oplus\bar E''_{x}),\Z)$ we get
$$U=(\partial \hat c_{x})_*\big(\big([\tilde Z]^{BM}\cdot(P'_{x}\times P''_{x})\big)\otimes \Sg_{N_W^Q}\big)\ .
$$

Therefore to complete the proof it suffices to show that
\begin{equation}\label{toshow}
(\partial \hat c_{x})_*\big(\big([\tilde Z]^{BM}\cdot(P'_{x}\times P''_{x})\big)\otimes \Sg_{N_W^Q}\big)=J^{E''_{x}}_{E'_{x}} \big([\tilde Z]^{BM} \cdot (P'_{x}\times P''_{x})  \big)\ .
\end{equation}

Taking into account (\ref{hatc}) the map $\partial \hat c_{x}$ decomposes as 
$\partial \hat c_{x}=q_{x}\circ G_{x}$,
 where 
\begin{itemize}
\item $G_{x}:(S(E'_{x})\times  S(E''_{x}))\times_{\beta} S(\C\oplus\R) \to (S(E'_{x})\times  S(E''_{x}))\times_{\alpha'} \P^1_\C$
is the bundle isomorphism induced by the  diffeomorphism $g:S(\C\oplus\R)\to \P^1_\C$
given  by 
$$g(u,t)=\left\{
\begin{array}{ccl}
\left[(1-t)^{\frac{1}{2}},  { \bar u}{(1-t)^{-\frac{1}{2}}} \right] &{\rm  for} & t\ne \phantom{-} 1 \vspace{1mm}\\
\left[ {u}{(1+t)^{-\frac{1}{2}}} , 
(1+t)^{\frac{1}{2}}\right]
&{\rm for }& t\ne -1
\end{array}
\right.   \forall (u,t)\in S(\C\oplus\R)\ .
$$
This diffeomorphism is orientation preserving if $S(\C\oplus\R)$ is endowed with the standard orientation (as boundary of $D^3$) and $\P^1_\C$ with the complex orientation; it maps $(0_\C,1)$ to $\infty\in\P^1_\C$, $(0_\C,-1)$ to $0\in\P^1_\C$, and is equivariant with respect to the $S^1\times S^1$-actions $\beta$, $\alpha'$ given by the formulae (\ref{beta}), (\ref{alpha'}).
\item $q_{x}:(S(E'_{x})\times  S(E''_{x}))\times_{\alpha'} \P^1_\C\to \P(E'_{x}\oplus\bar E''_{x})$ is defined by
$$q_{x}([(a',a''),[z_0,z_1]]):=[z_0 a', z_1 \bar a'']\ .
$$
\end{itemize}
It suffices to recall that, by (\ref{with-alpha'}),  one has
$$(S(E'_{x})\times  S(E''_{x}))\times_{\alpha'} \P^1_{\C}=
\P(p_1^*(\Theta_{E'_{x}})\oplus p_2^*(\bar\Theta_{E''_{x}})\ ,$$
 and to note that, via this identification, one has  $q_{x}=q_{E'_{x}}^{E''_{x}}$.
\end{proof}
Using our explicit formula for the join morphism (see Remark \ref{NewJoinOnGen}), Proposition \ref{partialZ}, and Remark \ref{P'P''}, we get

\begin{co}\label{partialQ} With the notations introduced above, suppose $r'+r''>0$. The boundary of the fundamental class   $[Q^s_f]^{BM}$ of the whole quotient $Q^s_f$ is given by 
$$\delta([Q^s_f]^{BM})=(-1)^{r''-1}[C]\otimes  w_{r'+r''-1}=(-1)^{r''-1}[C]\otimes  [\P(E'_{x}\oplus\bar E''_{x})]\ .
$$
\end{co}

Note that this formula holds even when $r''=0$, or $r'=0$, but in these cases one uses Remark \ref{P'P''}. When $r''=0$ ($r'=0$), one has $Q^s_f= P'_-$ (respectively $Q^s_f= P''_+$).
%
%
%
\begin{ex}\label{ex3}
Suppose $r''=1$.  Then
\begin{equation}\label{example1}
\delta([P'_-]^{BM}))=-[C]\otimes w_1\ ,\ \delta([P''_+]^{BM}))=[C]\otimes 1\ ,\ \delta([Q^s_f]^{BM})=[C]\otimes w_2
\end{equation}
In this case  $Q_0$ can be identified with $E'\otimes E''$ (hence is smooth), $Q$ is the blow up of $Q_0$ at the curve $B$, $Q^\st_f$ is the blow up of $Q_0\setminus C$ at the curve $B_-$, $c:Q\to Q_0$, $c_2:Q^\st_f\to Q_0\setminus C$ are  the corresponding contraction maps, the projective bundles $P'$,  $P'_-$ are the corresponding exceptional divisors, and $P''_+=B_+$.
Let $Z\subset Q^s_f$ be an {\it irreducible} hypersurface. Taking into account the position of $Z$ with respect to $P'_-$,   the following cases can occur:
\begin{enumerate}[1.]
\item $Z=P'_-$ (equivalently $\tilde Z=P'$). In this case $\delta([Z]^{BM})=-[C]\otimes w_1$.
\item $Z\ne P'_-$ (equivalently $\tilde Z\ne P'$).
 In this, denoting by $\deg(Z\cap P'_x)$ the degree of $ Z\cap P'_x$ in the projective space  $P'_x$ for generic $x\in B_-$, we'll have
 \begin{equation}\label{example2}
\delta([Z]^{BM})=\deg(Z\cap P'_x) [C]\otimes w_{r'-1}\ .
\end{equation}
  Taking into account the position of the divisor $c(\tilde Z)$ with respect to $B$ (in $Q_0$), two subcases can occur:
\begin{enumerate} 
\item $c(\tilde Z)\cap B$ is 0-dimensional. This implies that $\tilde Z$ does not intersect the generic fiber of $P'$, hence $\deg(Z\cap P'_x)=0$, and $\delta([Z]^{BM})=0$.
\item $c(\tilde Z)$ contains $B$. This implies that $\tilde Z$ intersects all the fibers of the projective bundle $P'$. Since we are in case 2. the intersection with the generic fiber must be transversal, hence in this case we have $\deg(Z\cap P'_x)>0$. Note that in this case  $Z$ contains $P''_+=B_+$.\end{enumerate}
\end{enumerate}
\end{ex}

\section{Applications}

The following simple result shows how Theorem \ref{partialZ} is applied.

\begin{pr}\label{identity}
Let $(X,g)$ be a Gauduchon surface with $b_1(X)=1$, $p_g(X)=0$, $(E,h)$ a Hermitian rank 2-bundle on $X$ with $c_1(E)\not\in 2H^2(X,\Z)$, and ${\cal D}$ a holomorphic structure on $\det(E)$. Suppose that
\begin{enumerate}[1.]
\item All reductions in ${\cal M}^\pst$ are regular,
\item The moduli space ${\cal M}^\pst$ associated with the data $(X,g,E,h,{\cal D})$ is compact. 
\end{enumerate}
For any Donaldson cohomology class $\nu\in H^{k-1}({\cal B}^*,\Q)$ and  any $\xi\in H_k^{BM}({\cal M}^\st,\Q)$ we have
$$\sum_{\lambda\in {\cal D}ec(E)} \langle \resto{\nu}{{\cal P}_\lambda},\delta_\lambda\xi\rangle =0\ ,
$$
where $\delta_\lambda\xi$ denotes the homological boundary of $\xi$ in $H_{k-1}({\cal P}_\lambda,\Q)$.\end{pr}
\begin{proof} Denote by $\hat {\cal M}^\pst$ the space obtained by blowing up all circles $C_\lambda$ in ${\cal M}^\pst$.
Using the   Borel-Moore long exact sequence
associated with the open embedding ${\cal M}^\st \hookrightarrow \hat {\cal M}^\pst$ \cite[Section 1.6]{BH}, and denoting by $j:\partial\hat {\cal M}^\pst\hookrightarrow \hat {\cal M}^\pst$ the inclusion map, we get $j_*(\delta \xi)=0$. Therefore
$$\langle \resto{\nu}{\partial \hat {\cal M}^\pst}, \delta \xi\rangle=\langle \resto{\nu}{\hat {\cal M}^\pst}, j_*(\delta \xi)\rangle=0\ .
$$
Here we used essentially that $\partial \hat {\cal M}^\pst$, $\hat {\cal M}^\pst$ are compact, hence their Borel-Moore homology coincides with their usual homology.  It suffices to take into account that $\partial \hat {\cal M}^\pst=\coprod_{\lambda\in{\cal D}ec(E)} {\cal P}_\lambda$.
\end{proof}

We can now apply our results to the geometric problem formulated in the introduction concerning the Zariski closures of the extension families ${\cal P}^c_\varepsilon$.  We are interested in the moduli space intervening in our program for proving the existence of curves on class VII surfaces. The following result will be used in \cite{Te6}. We use the notations introduced in section \ref{intro}. 

\begin{pr}\label{application}
Let $X$ be a minimal class VII surface with $H_1(X,\Z)\simeq\Z$ and $b_2(X)=3$, which is not an Enoki surface. Let $g$  be a Gauduchon metric on $X$ with $\deg_g({\cal K})<0$. Then
\begin{enumerate}[1.]
\item The component  ${\cal M}_0\in\pi_0({\cal M}^\pst)$ of ${\cal P}^0_\varepsilon$   contains all four circles of reductions,
\item Let $i\in\Ig$. The Zariski closure of ${\cal P}^{e_i}_\varepsilon$ in ${\cal M}^\st$ has pure dimension 2, does not intersect ${\cal P}^0_\varepsilon$,  and contains the curve ${\cal P}^{e_I}_\varepsilon$ for a subset $I\subset\Ig$ of cardinal 2.
\end{enumerate}
\end{pr}
\begin{proof}  Put $\lambda_0:=\{0,e_\Ig\}$, and, for $k\in\Ig$, put $\lambda_k:=\{e_k,e_{\Ig\setminus\{k\}}\}$. One has $r'_0=3$, $r''_0=0$, $r'_{e_k}=2$, $r''_{e_k}=1$. \\
\\
 1. We apply Proposition \ref{identity} to the analytic set ${\cal M}_0^\st:={\cal M}_0\cap {\cal M}^\st$ taking $\nu:=\mu(h)\cup \mu (u)$ for classes   $h\in H_1(X,\Z)$, $u\in H_2(X,\Z)$. By Corollary \ref{partialQ} we have
$$\delta[\Psi_0^{-1}({\cal M}_0^\st)]=-[C_0]\otimes [\P(E'_{0,y_0})]\ ,\ \delta[\Psi_{e_k}^{-1}({\cal M}_0^\st)]=a_k[C_{e_k}]\otimes [\P(E'_{e_k,y_k}\oplus \bar E''_{e_k,y_k})]\ ,
$$
where 
$a_k:=\left\{\begin{array}{ccc}
1 &\rm if & C_{\lambda_k}\subset {\cal M}_0\\
0 & \rm if & C_{\lambda_k}\not \subset {\cal M}_0
\end{array}\right.$, $y_0\in C_0$ and $y_k\in C_{e_k}$.
Using Theorem \ref{hol-model}, and formula (\ref{gamma}) we obtain
$$\langle\gamma_X,h\rangle\big\{\big\langle e_\Jg, u\big\rangle +\sum_k a_k \big\langle e_k -e_{\Jg\setminus\{k\}}, u\big\rangle\big\}=0\ \forall u\in H_2(X,\Z)\ ,\ \forall h\in H_1(X,\Z)\ ,
$$
Therefore $ e_\Jg   +\sum_k a_k  ( e_k -e_{\Jg\setminus\{k\}})=0$, which holds only when $a_k=1$ for all $k\in\Jg$.\vspace{1mm}\\
2. Using the fact that $X$ is minimal one can prove as in \cite{Te5} that $\P_y\subset {\cal M}^\st$ is a projective plane for any $y\in \Pic^0(X)_{(-\infty,\kg)}$, the union 
$${\cal P}^0_\infty:=\union_{y\in \Pic^0(X)_{(-\infty,\kg)}}\P_y$$
 is disjoint, and gives a {\it Zariski} open subset of ${\cal M}^\st_0$ (see \cite{Te6} for details). ${\cal P}^0_\infty$ is a $\P^2_\C$-bundle over the punctured disk $\Pic^0(X)_{(-\infty,\kg)}$ via the natural projection. One can also check easily that ${\cal P}^{e_i}_\varepsilon\cap   {\cal P}^0_\infty=\emptyset$. On the other hand, taking into account 1., we get ${\cal P}^{e_i}\subset {\cal M}^\st_0$. This shows that ${\cal P}^{e_i}$ is contained in the complement ${\cal M}^\st_0\setminus {\cal P}^0_\infty$, which is Zariski closed and nowhere dense. This proves that the Zariski closure $Z$ of ${\cal P}^{e_i}_\varepsilon$ in ${\cal M}^\st$ has pure dimension 2, and does not intersect ${\cal P}^0_\varepsilon$. For the last claim,  let ${\cal I}_k$ be the set of irreducible components of $Z\cap {\cal O}_{\lambda_k}$ (as hypersurface of ${\cal O}_{\lambda_k}$). 
Taking $\nu:= \mu(h)$ in Proposition \ref{identity}, and using the formulae explained in Example \ref{ex3}, we get
\begin{equation}\label{sum}
0=\langle \gamma_X,h\rangle\sum_{k\in\Ig} \sum_{Y\in {\cal I}_k} \deg(\Psi_{e_k}^{-1}(Y)\cap P'_{e_k,z_k})\ ,
\end{equation}
where $z_k\in B_{e_k,-}$, and we agree to write $\deg(\Psi_{e_k}^{-1}(Y)\cap P'_{e_k,z_k})=-1$ if $\Psi_{e_k}^{-1}(Y)=P'_{e_k,-}$.
Supposing that ${\cal O}_{\lambda_i}$ is sufficiently small, we'll have ${\cal P}^{e_i}_\varepsilon\cap {\cal O}_{\lambda_i}\in {\cal I}_i$. This intersection corresponds  to $P'_{e_i,-}$ via $\Psi_{e_i}$. Therefore at least a term on the right  in (\ref{sum}) is -1, hence there exists $k\in\Jg$ and $Y\in {\cal I}_k$ such that $\deg(\Psi_{e_k}^{-1}(Y)\cap P'_{e_k,x}))>0$. As explained in Example \ref{ex3}, this implies $P''_{e_k,+}\subset \Psi_{e_k}^{-1}(Y)$, i.e., ${\cal P}^{e_{\Ig\setminus\{k\}}}_\varepsilon\subset Y$.
\end{proof}

\end{document}